\documentclass[11pt,a4paper,english]{amsart}
\usepackage[english]{babel}
\usepackage{amsthm}
\usepackage{amsmath,color}
\usepackage{amssymb}
\usepackage{graphicx}
\usepackage{dsfont}
\usepackage[bf]{caption}
\usepackage{hyperref} 
\usepackage{parskip}

\newtheorem{thm}{Theorem}
\newtheorem{coro}[thm]{Corollary}
\newtheorem{lem}[thm]{Lemma}
\newtheorem{prop}[thm]{Proposition}
\newtheorem{rk}[thm]{Remark}

\def\RR{\mathbb{R}}
\def\ZZ{\mathbb{Z}}

\def\EE{\mathbb{E}}

\newcommand{\p} {\textnormal{\textsf{P}}}
\newcommand {\cro}[1] {\left[ {#1} \right]}
\def \ind {\hbox{ 1\hskip -3pt I}}

\newcommand {\va}[1] {\left| {#1} \right|}

\begin{document}

\title[Range of the Campanino and P\'etritis random walk]
{On the range of the Campanino and P\'etritis\\* random walk}
\author{Nadine Guillotin-Plantard} 
\address{Institut Camille Jordan, CNRS UMR 5208, Universit\'e de Lyon, Universit\'e Lyon 1, 43, Boulevard du 11 novembre 1918, 69622 Villeurbanne, France.}
\email{nadine.guillotin@univ-lyon1.fr}

\author{Fran\c{c}oise P\`ene}
\address{Universit\'e de Brest and IUF,
LMBA, UMR CNRS 6205, 29238 Brest cedex, France}
\email{francoise.pene@univ-brest.fr}

\keywords{Range; random walk in random scenery; local limit theorem; local time; stable process\\
This research was supported by the french ANR project MEMEMO2}

\begin{abstract}
We are interested in the behaviour of the range of the Campanino and P\'etritis random walk \cite{CP},
namely a simple random walk on the lattice $\mathbb Z^2$ with random orientations of the horizontal layers. 
We also study the range of random walks in random scenery, from which the asymptotic behaviour of  the range of the first coordinate 
of the Campanino and P\'etritis random walk can be deduced.
\end{abstract}

\maketitle

\section{Introduction and main results}
We consider the random walk on a randomly oriented lattice $M=(M_n)_n$ 
considered by Campanino and P\'etritis \cite{CP}. It is a particular example of transient 2-dimensional random
walk in random environment. We fix a $p\in(0,1)$ corresponding to the probability for
$M$ to stay on the same horizontal line. 
The environment is given by a sequence $\epsilon=(\epsilon_k)_{k\in\mathbb Z}$
of i.i.d. (independent identically distributed) centered random variables with
values in $\{\pm 1\}$ and defined on the probability space
$(\Omega,\mathcal T,\p)$. Given $\epsilon$, $M$ is a 
closest-neighbourghs random walk on $\mathbb Z^2$ starting
from $0$ (i.e. $\mathbb P^\epsilon(M_0=0)=1$) and with transition probabilities
$$\mathbb P^\epsilon(M_{n+1}=(x+ \epsilon_y,y)|M_n=(x,y))=p,\quad
   \mathbb P^\epsilon(M_{n+1}=(x,y\pm 1)|M_n=(x,y))=\frac{1-p}2. $$
We will write $\mathbb P$ for the annealed expectation, that is the
integration of $\mathbb P^\epsilon$ with respect to $\p$. 
In the papers \cite{GPN} and \cite{BFFN1} respectively, a functional limit theorem and a local limit theorem were proved for the random walk $M$ under the annealed measure $\mathbb P$.
In this note we are interested in the asymptotic behaviour of the range ${\mathcal R}_n$ of $M$, i.e.
of the number of sites visited by $M$ before time $n$: 
$${\mathcal R}_n:=\#\{M_0,\dots M_{n}\}.$$
Since we know (see \cite{CP,BFFN1}) that $M$ is transient for almost every environment $\epsilon$,
it is not surprising that ${\mathcal R}_n$ has order $n$. More precisely we prove the following
result.
\begin{prop}\label{cvpsrange}
The sequence $({\mathcal R}_n/n)_n$ converges $\mathbb P$-almost surely to $\mathbb P[M_j\ne 0,\ \forall j\geq 1]$.
\end{prop}
We observe that the almost sure convergence result stated for the annealed
probability $\mathbb P$ implies directly the same convergence result 
for the quenched probability $\mathbb P^\epsilon$
for $\p$-almost every $\epsilon$.

Since ${\mathcal R}_n\le n+1$, due to the Lebesgue dominated convergence theorem,
we directly obtain the next result.
\begin{coro}
We have
$\mathbb E[{\mathcal R}_n]\sim n \mathbb P[M_j\ne 0,\ \forall j\geq 1]$ and
$\mathbb E^\epsilon[{\mathcal R}_n]\sim n \mathbb P[M_j\ne 0,\ \forall j\geq 1]$
for $\p$-almost every $\epsilon$.
\end{coro}
This last result contradicts the result expected by Le Ny in \cite{LeNy}
for the behaviour of the quenched expectation.
The main  
difficulty of this model is that $M$ has stationary increments under the
annealed probability $\mathbb P$ and is a Markov chain under 
the quenched probability $\mathbb P^\epsilon$ for $\p$-almost every $\epsilon$
but $M$ is not a Markov chain with stationary increments 
(neither for $\mathbb P$ nor for $\mathbb P^\epsilon$).
This complicates seriously our study.
\begin{rk}
For $\p$-almost every $\epsilon$, $(M_n)_n$ is a transient
Markov chain with respect to $\mathbb P^\epsilon$, hence
$\mathbb P^\epsilon[M_j\ne 0,\ \forall j\geq 1]=1/\sum_{n\ge 0}\mathbb P^\varepsilon(M_n=0)$ 
and
$\mathbb P[M_j\ne 0,\ \forall j\geq 1]=\mathbb E\left[1/\sum_{n\ge 0}\mathbb P^\varepsilon(M_n=0)\right]$. 
\end{rk}

The Campanino and P\'etritis random walk is closely related to Random Walks in Random Scenery (RWRS). This fact was first noticed in \cite{GPN}. More precisely the first coordinate of the Campanino and P\'etritis random walk can be viewed as a generalized RWRS, the second coordinate being a lazy random walk on $\mathbb Z$ (see Section 5 of \cite{BFFN1} for the details). The main point is that the range of the first coordinate of the Campanino and P\'etritis random walk can easily be deduced from the following results about the range of random walks in random scenery. Let us recall the definition of the RWRS.
Let $\xi:=(\xi_y,y\in \ZZ)$ and $X:=(X_k,k\ge 1)$ 
be two independent sequences of independent
identically distributed random variables taking their values in $\ZZ$. 
The sequence $\xi$ is called the {\it random scenery}. 
The sequence $X$ is the sequence of increments of the {\it random walk}  
$(S_n, n \geq 0)$
defined by $S_0:=0$ and  $S_n:=\sum_{i=1}^{n}X_i$, for $n\ge 1$. 
The {\it random walk in random scenery} (RWRS) $Z$ is 
then defined by
$$Z_0:=0\ \mbox{and}\ \forall n\ge 1,\ Z_n:=\sum_{k=1}^{n}\xi_{S_k}.$$
Denoting by  
$N_n(y)$ the local time of the random walk $S$~:
$$N_n(y):=\#\{k=1,...,n\ :\ S_k=y\} \, ,
$$
it is straightforward to see that 
$Z_n$ can be rewritten as $Z_n=\sum_y\xi_y N_n(y)$.

As in \cite{KS}, the distribution of $\xi_0$ is assumed to belong to the normal 
domain of attraction of a strictly stable distribution 
$\mathcal{S}_{\beta}$ of 
index $\beta\in (0,2]$, with characteristic function $\phi$ given by
$$\phi(u)=e^{-|u|^\beta(A_1+iA_2 \text{sgn}(u))}\quad u\in\mathbb{R},$$
where $0<A_1<\infty$ and $|A_1^{-1}A_2|\le |\tan (\pi\beta/2)|$.
When $\beta > 1$, this implies that $\EE[\xi_0] = 0$.
When $\beta = 1$, we assume
the symmetry condition
$\sup_{ t > 0} \va{\EE\cro{\xi_0 \ind_{\{\va{\xi_0} \le t\}}}} < +\infty \, $.\\
Concerning the random walk, the distribution of $X_1$ is
assumed to belong to the normal basin of attraction of a stable
distribution ${\mathcal S}'_{\alpha}$ with index $\alpha\in (0,2]$, with characteristic function $\psi$ given by
$$\psi(u)=e^{-|u|^\alpha(C_1+iC_2 \text{sgn}(u))}\quad u\in\mathbb{R},$$
where $0<C_1<\infty$ and $|C_1^{-1}C_2|\le |\tan (\pi\alpha/2)|$. In the particular case where $\alpha =1$, we assume that $C_2=0$. Moreover we assume that the additive group $\mathbb Z$ is generated by the support of the distribution of $X_1$.

\noindent Then the following weak convergences hold in the space  of 
c\`adl\`ag real-valued functions 
defined on $[0,\infty)$ endowed with the 
Skorohod $J_1$-topology~:
$$\left(n^{-\frac{1}{\alpha}} S_{\lfloor nt\rfloor}\right)_{t\geq 0}   
\mathop{\Longrightarrow}_{n\rightarrow\infty}
^{\mathcal{L}} \left(Y(t)\right)_{t\geq 0}\, ,$$
$$   \left(n^{-\frac{1}{\beta}} 
\sum_{k=0}^{\lfloor nx\rfloor}\xi_{k}\right)_{x\ge 0}
   \mathop{\Longrightarrow}_{n\rightarrow\infty}^{\mathcal{L}} 
\left(U(x)\right)_{x\ge 0}\quad\mbox{and}\quad
\left(n^{-\frac{1}{\beta}} 
\sum_{k=\lfloor -nx\rfloor}^{1}\xi_{k}\right)_{x\ge 0}
   \mathop{\Longrightarrow}_{n\rightarrow\infty}^{\mathcal{L}} 
\left(U(-x)\right)_{x\ge 0}$$
where $(U(x))_{x\ge 0}$,  $(U(-x))_{x\ge 0}$ and $(Y(t))_{t\ge 0}$ are three independent L\'evy processes such 
that $U(0)=0$, $Y(0)=0$, 
$Y(1)$ has distribution $\mathcal{S}'_{\alpha}$, $U(1)$ and $U(-1)$ 
have distribution  $\mathcal{S}_\beta$.
We will denote by $(L_t(x))_{x\in\mathbb{R},t\geq 0}$ a continuous version with compact support of the local time of the process $(Y(t))_{t\geq 0}$. Let us define
 $$\delta := 1-\frac{1}{\alpha}+ \frac{1}{\alpha \beta}.$$
In the case $\alpha\in (1,2]$ and $\beta\in (0,2]$, Kesten and Spitzer \cite{KS} proved 
the convergence in distribution of $(n^{-\delta} Z_{[nt]})_{t\ge 0}, n\geq 1$ (with respect to the $J_1$-metric), 
to a process $\Delta=(\Delta_t)_{t\geq 0}$ defined in this case by
$$\Delta_t : = \int_{\RR} L_t(x) \, d U(x).$$
This process $\Delta$ is called Kesten-Spitzer process in the literature.

When $\alpha\in(0,1)$ (when the random walk $S$ is transient) and $\beta\in(0,2]\setminus \{1\} $,
$(n^{-\frac 1\beta} Z_{[nt]})_{t\ge 0}, n\geq 1$ converges in distribution (with respect to the $M_1$-metric), 
to $(\Delta_t:=c_0U_t)_{t\geq 0}$ for some $c_0>0$.

When $\alpha=1$ and $\beta\in(0,2]\setminus \{1\} $,
$(n^{-\frac 1\beta}(\log n)^{\frac 1\beta-1} Z_{[nt]})_{t\ge 0}, n\geq 1$ converges in distribution (with respect to the $M_1$-metric), 
to $(\Delta_t:=c_1U_t)_{t\geq 0}$ for some $c_1>0$.

Hence in any of the cases considered above, $(Z_{\lfloor nt\rfloor}/a_n)_{t\ge 0}$ converges in distribution (with respect to the
$M_1$-metric) to some process $\Delta$, with
$$
a_n:=\left\{ \begin{array}{lll}
n^{1- \frac{1}{\alpha} +\frac{1}{\alpha\beta}} & \text{if} & \alpha\in (1,2]\\
n^{\frac{1}{\beta}} (\log n)^{1-\frac{1}{\beta}} & \text{if} & \alpha =1\\
n^{\frac{1}{\beta}} & \text{if} & \alpha \in (0,1).
\end{array}
\right.
$$
We are interested in the asymptotic behaviour of the range ${\mathcal R}_n$ of the RWRS $Z$, i.e.
of the number of sites visited by $Z$ before time $n$: 
$${\mathcal R}_n:=\#\{Z_0,\dots, Z_{n}\}.$$
In the case when the RWRS is transient, we use the same argument as for $(M_n)_n$ and obtain the same kind of result.
\begin{prop}\label{cvpsrange2}
Let $\alpha\in (0,2]$ and $\beta\in (0,1)$. Then, $({\mathcal R}_n/n)_n$ converges $\mathbb P$-almost surely to $\mathbb P[Z_j\ne 0,\ \forall j\geq 1]$.
\end{prop}
For recurrent random walks in random scenery, we distinguish the 
easiest case when $\xi_1$ takes its values in $\{-1,0,1\}$. In that case, $\beta=2$, $U$ is the standard real Brownian motion,
$$
a_n=\left\{ \begin{array}{lll}
n^{1- \frac{1}{2\alpha}} & \text{if} & \alpha\in (1,2]\\
\sqrt{n \log n} & \text{if} & \alpha =1\\
\sqrt{n} & \text{if} & \alpha \in (0,1)
\end{array}
\right.
$$
and the limiting process $\Delta$ is either the Kesten-Spitzer process (case $\alpha\in (1,2])$ or the real Brownian motion (case $\alpha\in(0,1]$). Remark that in any case the limiting process
is symmetric.
\begin{prop}\label{cvpsrange3}
If $\alpha\in(0,2]$ and if $\xi_1$ takes its values
in $\{-1,0,1\}$. Then
$$\frac{\mathcal R_n}{a_n}=\frac{\sup_{t\in[0,1]}Z_{\lfloor nt\rfloor}-\inf_{t\in[0,1]}Z_{\lfloor nt\rfloor}+1}{a_n} \stackrel{\mathcal L}{\longrightarrow} \sup_{t\in[0,1]}\Delta_t-\inf_{t\in[0,1]}\Delta_t$$
and
$$\lim_{n\rightarrow+\infty} \frac{\mathbb{E}[\mathcal R_n]}{a_n}=2\, \mathbb E\left[\sup_{t\in[0,1]}\Delta_t\right].$$
\end{prop}
We also study the asymptotic behaviour of the range of the first coordinate of the Campanino and P\'etritis random walk. 
Let $\mathcal R_n^{(1)}$ be the number of vertical lines visited by
$(M_k)_k$ up to time $n$,
i.e.
$$\mathcal R_n^{(1)}:=\#\{x\in\mathbb Z\ :\ \exists k=0,...,n,\ \exists y\in\mathbb Z\ :\ M_k=(x,y)\}.$$
Let us recall that it has been shown in \cite{GPN} that the first coordinate
of $M_{\lfloor nt\rfloor}$ normalized by $n^{\frac 34}$ converges in
distribution to $K_p\Delta_t^{(0)}$, where $K_p:=\frac p{(1-p)^{\frac 14}}$ and where $\Delta^{(0)}$ is the Kesten-Spitzer process $\Delta$ with $U$ and $Y$ two independent standard Brownian motions.
\begin{prop}[Range of the first coordinate of the Campanino and P\'etritis random walk]\label{cvpsrange3bis}
$(\mathcal R_n^{(1)}/n^{\frac 34})_n$ converges in distribution
to $K_p\left(\sup_{t\in[0,1]}\Delta_t^{(0)}-\inf_{t\in[0,1]}\Delta_t^{(0)}
\right).$
Moreover
$$\lim_{n\rightarrow+\infty} \frac{\mathbb{E}[\mathcal R_n^{(1)}]}{n^{\frac 34}}=2K_p\, \mathbb E\left[\sup_{t\in[0,1]}\Delta^{(0)}_t\right].$$
\end{prop}
Since the second coordinate of the Campanino and P\'etritis random walk is a true random walk, the asymptotic behaviour of its range
is well known \cite{LGR}.\\
The range of RWRS in the general case $\beta\in(1,2]$ is much more
delicate. Indeed, the fact that $\mathcal R_n$  is less than $\sup_{t\in[0,1]}Z_{\lfloor nt\rfloor}-\inf_{s\in[0,1]}Z_{\lfloor ns\rfloor}+1$ will only provide an upper bound; we use a separate argument to obtain the lower bound insuring that $\mathcal R_n$ has order $a_n$.
\begin{prop}\label{cvpsrange4}
Let $\alpha\in(0,2]$ and $\beta\in (1,2]$. Then
$$0<\liminf_{n\rightarrow +\infty}\frac{\mathbb E[\mathcal R_n]}{a_n}\leq \limsup_{n\rightarrow +\infty}\frac{\mathbb E[\mathcal R_n]}{a_n}<\infty.$$
\end{prop}
We actually prove that $\limsup_{n\rightarrow +\infty}\frac{\mathbb E[\mathcal R_n]}{a_n}\le \mathbb E[\sup_{t\in[0,1]}\Delta_t-\inf_{t\in[0,1]}\Delta_t]$.
The question wether $\lim_{n\rightarrow +\infty}\frac{\mathbb E[\mathcal R_n]}{a_n}=\mathbb E[\sup_{t\in[0,1]}\Delta_t-\inf_{t\in[0,1]}\Delta_t]$ or not is still open.\\*
The paper is organized as follows. Section \ref{Transient case} provides the proof of Propositions \ref{cvpsrange} and \ref{cvpsrange2}.
Section \ref{Recurrent case} is devoted to the proof of Propositions \ref{cvpsrange3}, \ref{cvpsrange3bis} and \ref{cvpsrange4}. 
\section{Behaviour of the range in transient cases}\label{Transient case}
Let $(\Omega,\mu,T)$ be an ergodic probability dynamical system
and let $f:\Omega\rightarrow\mathbb Z^d$ be a measurable function.
We consider the process $(M_n)_{n\ge 0}$ defined by
$M_n=\sum_{k=0}^{n-1}f\circ T^k$ for $n\geq 1$ and $M_0=0$.
Now we assume that 
$\sum_{n\ge 0}\mathbb P(M_n=0)< +\infty$, so $(M_n)_n$ is transient. 
Let ${\mathcal R}_n$ be the range of $(M_n)_n$, that is
${\mathcal R}_n:=\#\{M_0,...,M_{n}\}$.
\begin{prop}\label{proptransient}
Assume that $\mathbb P(M_n=0)=O(n^{-\theta})$ for some
$\theta>1$. Then $\lim_{n\rightarrow +\infty}{\mathcal R}_n/n= \mu(M_j\ne 0,\ \forall j\ge 1)$, $\mu$-almost surely.
\end{prop}
\begin{proof}
It is worth noting that 
$${\mathcal R}_n=1+\sum_{k=0}^{n-1}\mathbf 1_{\{M_{k+j}\ne M_k,\, \forall j=1,...,n-k\}}.$$
Indeed $M_{k+j}\ne M_k,\, \forall j=1,...,n-k$ means that the site $M_k$ visited at time $k$
is not visited again before time $n$. We define now
$${\mathcal R}'_n:=1+\sum_{k= 0}^{n-1}1_{\{M_{k+j}-M_k\ne 0,\, \forall j\ge 1\}}.$$
We first prove the almost sure convergence of $({\mathcal R}'_n/n)_n$.
To this end, we observe that ${\mathcal R}'_n$ can be rewritten
$${\mathcal R}'_n=1+\sum_{k= 0}^{n-1}1_{\{M_{j}\ne 0,\, \forall j\ge 1\}}\circ T^k.$$
By ergodicity of $T$, $({\mathcal R}'_n/n)_n$ converges almost surely to $\mathbb P[M_j\ne 0,\ \forall j\geq 1]$.
\\
Now let us estimate ${\mathcal R}_n-{\mathcal R}'_n$.
We have
\begin{eqnarray*}
\Vert {\mathcal R}_n-{\mathcal R}'_n\Vert_1&=&\mathbb E[{\mathcal R}_n-{\mathcal R}'_n]\\
  &\le&\sum_{k= 0}^{n-1}\mathbb P(\exists j\ge n-k,\ M_{k+j}-M_k=0)\\
  &\le&\sum_{k= 0}^{n-1}\mathbb P(\exists j\ge n-k,\ M_{j}=0)\\
  &\le&\sum_{k= 1}^{n}\sum_{j\ge k}\mathbb P(M_{j}=0)\\
  &\le&\sum_{k= 1}^{n}\sum_{j\ge k} Cj^{-\theta}\\
  &=&\left\{\begin{array}{lll}
             O(n^{2-\theta}) & \text{ when } & 1<\theta <2\\
             O(\log n) &  \text{ when } & \theta = 2\\
             O(1) & \text{ when } & \theta > 2
            \end{array}
            \right.
\end{eqnarray*}
using the stationarity of the increments of $(M_n)_n$.
Hence, when $1<\theta <2$,
$\Vert ({\mathcal R}_n-{\mathcal R}' _n)/n\Vert_1=O(n^{1-\theta})$. Let $\gamma>0$
be such that $\gamma(\theta-1)>1$. 
Due to the Borel-Cantelli Lemma,
$(({\mathcal R}_{k^\gamma}-{\mathcal R}'_{k^\gamma})/k^\gamma)_k$ to 0,
and so $({\mathcal R}_{k^\gamma}/k^\gamma)_k$ 
converges almost surely to 
$\mathbb P[M_j\ne 0,\ \forall j\geq 1]$. To conclude, we use the increase of $({\mathcal R}_n)_n$
which gives that
$$\frac{{\mathcal R}_{\lfloor n^{\frac 1\gamma}\rfloor^\gamma}}{n}\le \frac{{\mathcal R}_n}{n}\le \frac{{\mathcal R}_{\lceil {n^{\frac 1\gamma}}\rceil^\gamma}}{n}.$$
We conclude by noticing that $\lceil n^{\frac 1\gamma}\rceil^\gamma\sim n$ and
$\lfloor n^{\frac 1\gamma}\rfloor^\gamma\sim n$. 
\\
The cases $\theta =2$ and $\theta >2$ can be handled in a similar way. 
\end{proof}
\begin{proof}[Proof of Proposition \ref{cvpsrange}]
Let us consider $\Omega:=\{-1,1\}^{\mathbb Z}\times
\{-1,0,1\}^{\mathbb Z}$ and the transformation $T$ on $\Omega$ given by 
$T((\epsilon_k)_k,(\omega_k)_k)=((\epsilon_{k+\omega_0})_k,
  (\omega_{k+1})_k)$. This transformation preserves
the probability measure $\mu:=(\frac{\delta_1+\delta_{-1}}2)^{\otimes\mathbb Z}\otimes(p\delta_0+\frac{1-p}2\delta_1+\frac{1-p}2\delta_{-1})^{\otimes\mathbb Z}$ and is ergodic (see for instance \cite{KMC}, p.162).
\\
We also set $f((\epsilon_k)_k,(\omega_k)_k)=(\epsilon_0,0)$
if $\omega_0 = 0$, $f((\epsilon_k)_k,(\omega_k)_k)=(0,\omega_0)$
otherwise.
\\
We observe that $(M_j)_{j\geq 1}$ has the same distribution under $\mathbb P$
as $(\sum_{k=0}^{j-1} f\circ T^j)_{j\geq 1}$ under $\mu$.
We conclude by Proposition \ref{proptransient} since we know from \cite{BFFN1} that
$\mathbb P(M_{n}=0)=O(n^{-\theta})$ with $\theta=5/4$.
\end{proof}
\begin{proof}[Proof of Proposition \ref{cvpsrange2}]
We consider $\Omega:=\mathbb Z^{\mathbb Z}\times
\mathbb Z^{\mathbb Z}$ and the transformation $T$ on $\Omega$ given by 
$T((\alpha_k)_k,(\epsilon_k)_k)=((\alpha_{k+1})_k,
  (\epsilon_{k+\alpha_0})_k)$. This transformation preserves
the probability measure $\mu:=(\mathbb P_{S_1})^{\otimes\mathbb Z}\otimes (\mathbb P_{\xi_1})^{\otimes\mathbb Z}$.
This time we set $f((\alpha_k)_k,(\epsilon_k)_k)=\epsilon_0$.
With these choices, $(Z_j)_{j\geq 1}$ has the same distribution under $\mathbb P$
as $(\sum_{k=1}^j f\circ T^j)_{j\geq 1}$ under $\mu$.
Again we conclude thanks to  Proposition \ref{proptransient}, to the
ergodicity of $T$ (see for instance \cite{KMC}, p.162) and to the local limit theorems established in \cite{BFFN1} (Theorems 1 and 2) and \cite{FFN} (Theorem 3).
\end{proof}
\section{Range of recurrent random walks in random scenery}\label{Recurrent case}
In this section we prove Propositions  \ref{cvpsrange3},  \ref{cvpsrange3bis} and \ref{cvpsrange4}.
We write $M_n^{(1)}$ for the first coordinate of the
Campanino and P\'etritis random walk $M_n$.
\\
For Propositions  \ref{cvpsrange3},  \ref{cvpsrange3bis},
we observe that 
${\mathcal R}_n= \max_{0\le k\le n}Z_k-\min_{0\le k\le n} Z_k +1$
and ${\mathcal R}_n^{(1)}= \max_{0\le k\le n}M_k^{(1)}-\min_{0\le k\le n} M_k^{(1)} +1$ whereas for 
Proposition \ref{cvpsrange4}, we only have
${\mathcal R}_n\le \max_{0\le k\le n}Z_k-\min_{0\le k\le n} Z_k +1$. Hence the convergence of the means in Propositions \ref{cvpsrange3} and  \ref{cvpsrange3bis} and the upper bound in Proposition \ref{cvpsrange4} will come from lemmas \ref{LEM0} and \ref{LEM0bis} below. Let us start by the convergence
in distribution.
\begin{proof}[Proof of the convergences in distribution]
Due to the convergence for the $M_1$-topology of $((a_n^{-1} Z_{\lfloor nt\rfloor})_t)_n$ to $(\Delta_t)_t$ as $n$ goes to infinity,
we know (see Section 12.3 in \cite{Whitt}) that $(a_n^{-1}(\max_{0\le k\le n}Z_k-\min_{0\le \ell\le n}Z_{\ell}))_n$ converges in distribution
to $\sup_{t\in[0,1]}\Delta_t-\inf_{s\in[0,1]}\Delta_s $
as $n$ goes to infinity.
\\
Due to \cite{GPN}, $((M^{(1)}_{\lfloor nt\rfloor}/n^{\frac 34})_t)_n$ converges in distribution to $(K_p\Delta_t^{(0)})_t$ in the Skorohod space endowed with the $J_1$-metric. 
Hence $(n^{-\frac 34}(\max_{k=0,...,n}M_k^{(1)}-\min_{\ell=0,...,n}M_\ell^{(1)}))_n$ converges in distribution
to $K_p(\sup_{t\in[0,1]}\Delta_t^{(0)}-\inf_{s\in[0,1]}\Delta_s^{(0)})$.
\end{proof} 
\begin{lem}[RWRS]\label{LEM0}
Assume $\beta>1$, then 
$$\lim_{n\rightarrow +\infty}\frac{\mathbb E\left[\max_{k=0,...,n}Z_k\right]}{a_n}=\mathbb E\left[\sup_{t\in[0,1]}\Delta_t \right].$$
\end{lem}
\begin{lem}[First coordinate of the Campanino and P\'etritis random walk]\label{LEM0bis}
$$\lim_{n\rightarrow +\infty}\frac{\mathbb E\left[\max_{k=0,...,n}M_k^{(1)}\right]}{n^{\frac 34}}=K_p\mathbb E\left[\sup_{t\in[0,1]}\Delta_t^{(0)} \right].$$
\end{lem}
\begin{proof}[Proof of Lemma \ref{LEM0}]
As explained above, we know that $(a_n^{-1}\max_{0\le k\le n}Z_k)_n$ converges in distribution
to $\sup_{t\in[0,1]}\Delta_t $
as $n$ goes to infinity.
Now let us prove that this sequence is
uniformly integrable. 
To this end we will use the fact that, conditionally to
the walk $S$, the increments of $(Z_n)_n$ are centered and positively associated.
Let $\beta'\in(1,\beta)$ be fixed.
Due to Theorem 2.1 of \cite{Gong}, there exists some constant $c_{\beta'}>0$ such that
\begin{eqnarray*}
\mathbb E\left[\left|\max_{j=0,...,n}Z_j\right|^{{\beta'}}|S\right] &\le & \mathbb E\left[\max_{j=0,...,n} |Z_j|^{{\beta'}} |S\right]   \\
& \le& c_{\beta'} \mathbb E\left[|Z_n|^{{\beta'}}|S\right]\\
\end{eqnarray*}
so
\begin{eqnarray*}
\mathbb E\left[\left|\max_{j=0,...,n}Z_j\right|^{{\beta'}}\right] &=& \mathbb E\left[\mathbb E\left[\left|\max_{j=0,...,n}Z_j\right|^{{\beta'}}|S\right]\right]\\
& \le& c_{\beta'} \mathbb E\left[|Z_n|^{{\beta'}}\right].
  \end{eqnarray*}  
It remains now to prove that $\mathbb E[|Z_n|^{\beta'}]=O(a_n^{\beta'})$.
\\
\noindent Let us first consider the easiest case when the random scenery is square integrable that is $\beta=2$, then we take $\beta'=2$ in the above computations and observe that 
$\mathbb E\left[|Z_n|^{2}\right]=\mathbb E[\xi_0^2]\mathbb E[V_n]$, 
where $V_n$ is the number of self-intersections up to time $n$
of the random walk $S$, i.e.  
$V_n=\sum_x(N_n(x))^2=\sum_{i,j=1}^n{\mathbf 1}_{S_i=S_j}$. 
Usual computations (see Lemma 2.3 in \cite{bol}) give that
$$\mathbb E[V_n]=\sum_{i,j=1}^n\mathbb P(S_{i-j}=0) \sim c'(a_n)^2$$
and the result follows.
\\
\noindent When $\beta\in(1,2)$, let us define $V_n(\beta)$ as follows
$$V_n(\beta):=\sum_{y\in\mathbb Z}(N_n(y))^\beta.$$
Due to Lemma 2 of \cite{VBE},
\begin{equation}\label{EZ_nbeta'} 
\mathbb E\left[|Z_n|^{\beta'}\right] =\frac {\Gamma(\beta'+1)}{\pi}\sin\left(\frac{\pi\beta'} 2\right)\int_{\mathbb R}
\frac{1-\mbox{Re}(\varphi_{Z_n}(t))}{|t|^{\beta'+1}}\, dt,
\end{equation}
where $\varphi_{Z_n}$ stands for the characteristic function of $Z_n$,
which is given by
\begin{equation}\label{fonctioncar}
\forall t\in\mathbb R,\quad
\varphi_{Z_n}(t):=\mathbb E[e^{itZ_n}]=\mathbb E\left[\mathbb E\left[e^{itZ_n}|(S_k)_k\right]\right]
=\mathbb E\left[\prod_{y\in\mathbb Z}\varphi_\xi(tN_n(y))\right].
\end{equation}
Due to our assumptions on $\xi$, we know that 
$1-\varphi_\xi(u) = |u|^\beta(A_1+i A_2 \mbox{sgn} (u))(1+o(1))$ as $u$ goes to $0$.
Let $A,B>0$ be such that  $|1-\varphi_\xi(u)|<B|u|^\beta$ for every real number $u$ satisfying $|u|<A$.
Hence, for every $t$ such that $|t|<A(V_n(\beta))^{-\frac 1\beta}$, we have
$|tN_n(y)|\le A$ and so $|1-\varphi_\xi(tN_n(y))|\le B|t|^\beta(N_n(y))^\beta$ and
\begin{eqnarray*}
\left|1-\mbox{Re}\left(\prod_{y\in\mathbb Z}\varphi_\xi(tN_n(y))\right)\right|&\le&
\left|1-\left(\prod_{y\in\mathbb Z}\varphi_\xi(tN_n(y))\right)\right|\\
&\le&\sum_{y\in\mathbb Z}|1-\varphi_\xi(tN_n(y))|\\
&\le& B |t|^\beta V_n(\beta).
\end{eqnarray*}
Hence 
\begin{eqnarray}
\int_{|t|<A(V_n(\beta))^{-\frac 1\beta}}\frac{\left|1-\mbox{Re}\left(\prod_{y\in\mathbb Z}\varphi_\xi(tN_n(y))\right)\right|}{
|t|^{\beta'+1}}\, dt
&\le&\int_{|t|<A(V_n(\beta))^{-\frac 1\beta}}\frac{B |t|^\beta V_n(\beta)}{
|t|^{\beta'+1}}\, dt\nonumber\\
&\le&B V_n(\beta)\int_{|t|<A(V_n(\beta))^{-\frac 1\beta}}|t|^{\beta-\beta'-1}\, dt\nonumber\\
&\le&\frac{2\, A^{\beta-\beta'}B}{\beta-\beta'} (V_n(\beta))^{\frac{\beta'}{\beta}}\label{INT1}.
\end{eqnarray}
Moreover
\begin{eqnarray}
\int_{|t|\ge A(V_n(\beta))^{-\frac 1\beta}}\frac{\left|1-\mbox{Re}\left(\prod_{y\in\mathbb Z}\varphi_\xi(tN_n(y))\right)\right|}{
|t|^{\beta'+1}}\, dt
&\le&2\int_{|t|\ge A(V_n(\beta))^{-\frac 1\beta}}
|t|^{-\beta'-1}\, dt\nonumber\\          
&\le& 4 \beta' A^{-\beta'}(V_n(\beta))^{\frac {\beta'}\beta}.\label{INT2}
\end{eqnarray}
Putting together \eqref{EZ_nbeta'}, \eqref{fonctioncar}, \eqref{INT1}
and \eqref{INT2}, we obtain that there  exists some constant $C>0$ such that for every $n$
$$\mathbb E[|Z_n|^{\beta'}]\leq C \mathbb E\left[(V_n(\beta))^{\frac{\beta'}\beta}\right].$$
If $\alpha >1$, due to Lemma 3.3 of \cite{NadineClement}, we know that
$\mathbb E[V_n(\beta)]=O\left(a_n^{\beta}\right)$ and so
\begin{equation}\label{Vnbeta'}
\mathbb E\left[(V_n(\beta))^{\frac{\beta'}\beta}\right]=O\left(a_n^{\beta'}\right).
\end{equation}
If $\alpha \in (0,1]$, using H\"{o}lder's inequality, we have
$$\mathbb E[V_n(\beta)] \leq \mathbb E[R_n]^{1-\frac{\beta}{2}} \mathbb E[V_n]^{\frac{\beta}{2}}.$$
Now if $\alpha =1$, we know that $\mathbb E[R_n] \sim c\frac{n}{\log n}$ (see for instance Theorem 6.9, page 398 in \cite{LGR}) and $\mathbb E[V_n] \sim c n\log n$ so $\mathbb E [V_n(\beta)] = O\left(a_n^{\beta}\right)$
with $a_n = n^{\frac{1}{\beta}} (\log n)^{1-\frac{1}{\beta}}$. In the case $\alpha\in (0,1)$,the random walk is transient and the expectations of $R_n$ and $V_n$ behaves as $n$, we deduce that
 $\mathbb E [V_n(\beta)] = O\left(a_n^{\beta}\right)$
with $a_n = n^{\frac{1}{\beta}}$.\\*
We conclude that
$$    \lim_{n\rightarrow+\infty}\mathbb E\left[\max_{j=0,...,n}\frac{Z_j}{a_n}\right]
        = \mathbb E\left[\max_{t\in[0,1]}\Delta_t \right].$$
\end{proof}
\begin{proof}[Proof of Lemma \ref{LEM0bis}]
We know that $(n^{-\frac 34}\max_{k=0,...,n}M_k^{(1)})_n$ converges in distribution
to $K_p\sup_{t\in[0,1]}\Delta_t^{(0)}$.
To conclude, it is enough to prove that this sequence is
uniformly integrable. To this end we will prove that
it is bounded in $L^2$.
\\
Recall that the second coordinate of the Campanino and P\'etritis random walk is a random walk. Let us write it $(S_n)_n$.
Observe that 
$$M_n^{(1)}:=\sum_{k=1}^n \varepsilon_{S_k}\ind_{\{S_k=S_{k-1}\}}=\sum_{y\in\mathbb Z} \varepsilon_{y}\tilde N_n(y),$$
with $\tilde N_n(y):=\#\{k=1,...,n\ :\ S_k=S_{k-1}=y\}$.
Observe that $\tilde N$ is measurable with respect to the random walk $S$ and that $0\le \tilde N_n(y)\le N_n(y)$.
\\
Conditionally to
the walk $S$, the increments of $(M^{(1)}_n)_n$ are centered and positively associated. It follows from Theorem 2.1 of \cite{Gong} that 
\begin{eqnarray*}
\mathbb E\left[\left|\max_{j=0,...,n}M^{(1)}_j\right|^2|S\right] 
& \le& c_{2} \mathbb E\left[|M^{(1)}_n|^2|S\right]\\
& \le& c_{2} \sum_{y\in\mathbb Z}(\tilde N_n(y))^2 \le c_{2} V_n,
\end{eqnarray*}
where again $V_n=\sum_{y\in\mathbb Z}(N_n(y))^2$.
Therefore
$$
\mathbb E\left[\left|\max_{j=0,...,n}M^{(1)}_j\right|^2\right] \le  c_{2} \mathbb E[V_n].$$
Again the result follows from the fact that $\mathbb E[V_n]\sim c'n^{\frac 32}$.
\end{proof}

\begin{proof}[Proof of the lower bound of Proposition \ref{cvpsrange4}]
Let $\mathcal N_n(x):=\#\{k=1,...,n\,:\,Z_k=x\}$.
Applying the Cauchy-Schwarz inequality to 
$n= \sum_x  \mathcal N_n(x) \mathbf 1_{\{\mathcal N_n(x)>0\}} $, we obtain
$$n^2\le \sum_y \mathbf 1_{\{\mathcal N_n(y)>0\}}\, \sum_x(\mathcal N_n(x))^2 =\mathcal R_n\, \mathcal V_n,$$
with $\mathcal V_n=\sum_{x}(\mathcal N_n(x))^2
         =\sum_{i,j=1}^n\mathbf 1_{\{Z_i=Z_j\}}$
the number of self-intersections of $Z$ up to time $n$
and so using Jensen's inequality,
$$\frac{ \mathbb E[\mathcal R_n]}{a_n} \ge \frac{n^2}{a_n}  \mathbb E[(\mathcal V_n)^{-1}]\ge \frac{n^2}{a_n}{\mathbb E[\mathcal V_n]^{-1}}.$$
Moreover, using the local limit theorems for the RWRS proved in \cite{BFFN1,FFN},
\begin{eqnarray*}
\mathbb E[\mathcal V_n]&=&
    n+ 2 \sum_{1\leq i <j \leq n} \mathbb P(Z_{j-i}=0)  \sim    C' \frac{n^2}{a_n}.   
\end{eqnarray*}
Hence 
    $$\liminf_{n\rightarrow +\infty} \frac{\mathbb E[\mathcal R_n]}{a_n }\ge \frac 1{C'} >0.$$
\end{proof}    

\end{document}